\documentclass[a4paper,12pt]{article}
\usepackage{cmap}
\usepackage[cp1251]{inputenc}
\usepackage[english]{babel}
\usepackage[left=2cm,right=2cm,top=2cm,bottom=2cm]{geometry}
\usepackage{amssymb}
\usepackage{amsmath, amsthm}

\usepackage{hyperref}
\usepackage{enumerate}

\usepackage{graphicx}

\theoremstyle{plain}
\newtheorem{theorem}{Theorem}
\newtheorem{lemma}{Lemma}
\newtheorem{proposition}{Proposition}
\newtheorem{corollary}{Corollary}[theorem]

\theoremstyle{definition}

\newtheorem{remark}{Remark}
\newtheorem{definition}{Definition}

\begin{document}

\begin{center}\Large
\textbf{A generalization of Hall's theorem on hypercenter}
\normalsize

\bigskip

Viachaslau I. Murashka and Alexander F. Vasil'ev\footnote{This document is the results of the research project funded by  Belarusian Republican Foundation for Fundamental Research, project No. $\Phi20\text{P-}291$.}

\smallskip

Faculty of Mathematics and Technologies of Programming\\
246019 Francisk Skorina Gomel State University\\
Gomel, Belarus\\

\smallskip

emails: mvimath@yandex.ru and formation56@mail.ru.

\end{center}

\begin{abstract}
   Let $ \sigma$ be a partition of the set of all primes and $\mathfrak{F}$ be a hereditary formation. We described all  formations $\mathfrak{F}$  for which  the $\mathfrak{F}$-hypercenter and the intersection of  weak $K$-$\mathfrak{F}$-subnormalizers of all Sylow subgroups coincide in every group.  In particular the formation of  all $\sigma$-nilpotent groups has this property. With the help of our results we solve a particular case of L.A.~Shemetkov's problem about the intersection of $\mathfrak{F}$-maximal subgroups and the $\mathfrak{F}$-hypercenter.  As corollaries we obtained P.~Hall's and R.~Baer's classical results about the hypercenter. We proved that the non-$\sigma$-nilpotent graph of a group is connected and its diameter is at most~3.


\end{abstract}

\section{Introduction}

Throughout this paper, all groups are finite;    $G$ and $p$      always denote a finite group and     a prime respectively.
 The notion of the hypercenter of a group naturally appears with the definition of nilpotency of a group through upper central series.
R. Baer \cite{Baer1959} introduced and studied the analogue of hypercenter for the class of all supersoluble groups.
B. Huppert  \cite{Huppert1969} considered the $ \mathfrak{F}$-hypercenter where $\mathfrak{F}$ is a hereditary saturated formation.
L.A.~Shemetkov \cite{Shemetkov1974} extended the notion of $\mathfrak{F}$-hypercenter for graduated formations.
The $\mathfrak{F}$-hypercenter for formations of algebraic systems (including finite groups) was suggested in \cite{s6}.

Recall that a chief factor $H/K$ of  $G$ is called   $\mathfrak{X}$-\emph{central} (see \cite[p. 127--128]{s6}) in $G$ provided    $(H/K)\rtimes (G/C_G(H/K))\in\mathfrak{X}$. A normal subgroup $N$ of $G$ is said to be $\mathfrak{X}$-\emph{hypercentral} in $G$ if $N=1$ or $N\neq 1$ and every chief factor of $G$ below $N$ is $\mathfrak{X}$-central. The symbol $\mathrm{Z}_\mathfrak{X}(G)$ denotes the $\mathfrak{X}$-\emph{hypercenter} of $G$, that is, the product of all normal $\mathfrak{X}$-hypercentral in $G$ subgroups. According to \cite[Lemma 14.1]{s6} $\mathrm{Z}_\mathfrak{X}(G)$ is the largest normal $\mathfrak{X}$-hypercentral subgroup of $G$. If $\mathfrak{X}=\mathfrak{N}$ is the class of all nilpotent groups, then $\mathrm{Z}_\mathfrak{N}(G)$ is the hypercenter $\mathrm{Z}_\infty(G)$ of $G$.


One of the first characterizations of the hypercenter was obtained by P. Hall \cite{h2}. He proved that the hypercenter of
a group coincides with the intersection of normalizers of all its Sylow subgroups.
    P. Schmid \cite{Schmid2014} proved the analogue of Hall's result in profinite groups.
  There were generalizations of P. Hall's theorem in terms of intersections of normalizers of $\pi_i$-maximal subgroups \cite{Murashka2013} or Hall $\pi_i$-subgroups \cite{hu2019} where $\pi_i$ belongs to some partition $\sigma$ of $\mathbb{P}$ (see Corollaries \ref{corollary2} and \ref{corollary3}).

These results are the part of   research project in which the $\mathfrak{F}$-hypercenter and its generalizations are used as descriptors for characterising some structural
properties of the group. A useful tool that provides a suitable language in this direction is the theory of formations. Nowadays this project is actively developing  by many researchers (for example   \cite{Aivazidis2021,BES2014,hu2019} and \cite[Chapter 1]{Guo2015}).
As part of the above mentioned project, the aim of our paper is to describe all hereditary (not necessary saturated) formations $\mathfrak{F}$ for which the analogue of Hall's result holds for the $\mathfrak{F}$-hypercenter and find the applications of this result.
To formulate our results we need the following definitions.

Let $\mathfrak{F}$ be a formation.  O.H. Kegel \cite{Kegel1978} introduced the formation generalization of subnormality. Recall \cite[Definition 6.1.4]{s9} that a subgroup $H$ of  $G$ is called $K$-$\mathfrak{F}$-\emph{subnormal} in $G$ if there is a chain of subgroups
$ H=H_0\subseteq H_1\subseteq\dots\subseteq H_n=G$
with $H_{i-1}\trianglelefteq H_i$ or $H_{i}/\mathrm{Core}_{H_{i}}(H_{i-1})\in\mathfrak{F}$ for all $i=1,\dots,n$. Denoted by $H\,K$-$\mathfrak{F}$-$sn\,G$. If $\mathfrak{F}=\mathfrak{N}$, then the notions of  $K$-$\mathfrak{F}$-subnormal and subnormal subgroups coincide.

R.W. Carter \cite{Carter1962} and C.J. Graddon \cite{Graddon1971} studied subnormalizers and $\mathfrak{F}$-sub\-nor\-ma\-lizers respectively.    Note that the subnormalizer of a Sylow subgroup always exists and coincides with its normalizer. For an arbitrary subgroup a subnormalizer or $\mathfrak{F}$-subnormalizer    may not exist. A. Mann \cite{h9} suggested the concept of  a weak subnormalizer that  always exists but may be not  unique.  A subgroup  $T$ of $G$ is called \emph{a weak subnormalizer} of $H$ in $G$  if $H$ is subnormal in $T$ and if $H$ is subnormal in $M\leq G$ and $T\leq M$, then $T=M$. Here we introduced its generalization.

  \begin{definition}
    Let $\mathfrak{F}$ be a formation.  We shall call a subgroup  $T$ of $G$   \emph{a weak $K$-$\mathfrak{F}$-subnormalizer} of $H$ in $G$  if $H$ is $K$-$\mathfrak{F}$-subnormal in $T$ and if $H$ is $K$-$\mathfrak{F}$-subnormal in $M\leq G$ and $T\leq M$, then $T=M$.
   \end{definition}

It is clear that a weak $K$-$\mathfrak{F}$-subnormalizer always exists.
Note that the notions of weak subnormalizer and $K$-$\mathfrak{N}$-subnormalizer coincide. See \cite[A, Example 14.12]{s8} for an example of a group that has a subgroup without an unique weak subnormalizer.

Let $\sigma=\{\pi_i\mid i\in I\}$ be a partition of the  set $\mathbb{P}$ of all primes. According to A.N.~Skiba \cite{sp4}, a group $G$ is called $\sigma$-\emph{nilpotent}  if  $G$ has a normal Hall $\pi_i$-subgroup for every $i\in I$ with $\pi(G)\cap\pi_i\neq\emptyset$. The class of all $\sigma$-nilpotent groups is denoted by  $\mathfrak{N}_\sigma$. This class is a very interesting generalization of the class of nilpotent groups and widely studied, applied and are part of the actively developing nowadays $ \sigma$-method, i.e. the studying the properties of a group that depends on the given partition $ \sigma$ (for example, see
\cite{BKPP,CGZ,hu2019,Kazarin2011,Shemetkov1976,sp4}).
 The class $\mathfrak{N}$ of all nilpotent groups coincides with the class $\mathfrak{N}_\sigma$ for $\sigma=\{\{p\}\mid p\in\mathbb{P}\}$.

Recall \cite[Example 2.2.12]{s9} that $\underset{i\in I}\times\mathfrak{F}_{\pi_i}=(G=\underset{i\in I, \pi_i\cap\pi(G)\neq\emptyset}\times \mathrm{O}_{\pi_i}(G)\mid \mathrm{O}_{\pi_i}(G)\in\mathfrak{F}_{\pi_i})$  is a hereditary formation where     $\mathfrak{F}_{\pi_i}$ is a hereditary  formation with $\pi(\mathfrak{F}_{\pi_i})=\pi_i$ for all $i\in I$.
The main result of this paper is

   \begin{theorem}\label{gb}  Let  $\mathfrak{F}$ be a hereditary formation. The following statements are equivalent:

\begin{enumerate}[$(1)$]
 \item The intersection of all weak  $K$-$\mathfrak{F}$-subnormalizers of all cyclic primary subgroups coincides with the  $\mathfrak{F}$-hypercenter in every group.

 \item  The intersection of all weak  $K$-$\mathfrak{F}$-subnormalizers of all Sylow subgroups coincides with the  $\mathfrak{F}$-hypercenter in every group.

 \item There is a partition   $\sigma=\{\pi_i\mid i\in I\}$ of $\mathbb{P}$ such that the $\mathfrak{F}$-hypercenter coincides with the $\sigma$-nilpotent-hypercenter in every group.

 \item There is a partition   $\sigma=\{\pi_i\mid i\in I\}$ of $\mathbb{P}$ such that $\mathfrak{F}=\times_{i\in I}\mathfrak{F}_{\pi_i}$ where $\mathfrak{F}_{\pi_i}$ is a hereditary formation with $\pi(\mathfrak{F}_{\pi_i})=\pi_i$ and $\mathfrak{F}_{\pi_i}$ coincides with the class of all $ \pi_i$-groups for all $ i\in I$ with $|\pi_i|\geq 2$.
     \end{enumerate}   \end{theorem}

\begin{remark}
  As follows from \cite[Theorem]{Yi2015} formations from $(4)$ of Theorem  \ref{gb} are lattice formations.
\end{remark}


\begin{corollary}\label{gs}
  Let $\sigma=\{\pi_i\mid i\in I\}$ be a partition $\mathbb{P}$, $G$ be a group and $\mathcal{M}$ be a set of maximal $\pi_i$-subgroups of $G$, $\pi_i\in \sigma$, such that

\begin{enumerate}[(a)]

  \item if $H\in\mathcal{M}$, then $H^x\in\mathcal{M}$ for every $x\in G$;

  \item for every Sylow subgroup $P$ of $G$ there is $H\in\mathcal{M}$ with $P\leq H$.
\end{enumerate}
  Then the intersection of normalizers in $G$ of all subgroups from $\mathcal{M}$ is $\mathrm{Z}_{\mathfrak{N}_\sigma}(G)$.
\end{corollary}

\begin{corollary}[{\cite[Corollary 3.7]{Murashka2013}}]\label{corollary2}
Let $\sigma=\{\pi_i\mid i\in I\}$ be a partition $\mathbb{P}$.  The intersection of normalizers of all $\pi_i$-maximal subgroups of $G$, $\pi_i\in \sigma$, is $\mathrm{Z}_{\mathfrak{N}_\sigma}(G)$.
\end{corollary}

\begin{corollary}[{\cite[Theorem B(ii)]{hu2019}}]\label{corollary3}
Let $\sigma=\{\pi_i\mid i\in I\}$ be a partition $\mathbb{P}$.  Assume that a group $G$ posses a set $\mathcal{H}$ of Hall subgroups such that $\mathcal{H}$ contains exactly one Hall $\pi_i$-subgroup of $G$ with $\pi_i\cap\pi(G)\neq\emptyset$. Then
  $$\bigcap_{x\in G}\bigcap_{H\in\mathcal{H}} N_G(H^x)=\mathrm{Z}_{\mathfrak{N}_\sigma}(G). $$
\end{corollary}

\begin{corollary}[P. Hall \cite{h2}] The intersection of all normalizers of Sylow subgroups is the hypercenter in every group.
\end{corollary}

\begin{corollary} The intersection of all weak subnormalizers of cyclic primary subgroups is the hypercenter in every group.
\end{corollary}

\begin{corollary}[{\cite[Theorem 3.1(2)]{Murashka2013}}]\label{cor15} Let $\sigma=\{\pi_i\mid i\in I\}$ be a partition $\mathbb{P}$. A $ \pi_i$-element belongs to $\mathrm{Z}_{\mathfrak{N}_\sigma}(G)$ iff its permutes with all $ \pi_i'$-elements of a group $G$.
\end{corollary}

\begin{corollary}[R. Baer {\cite[5, Theorem 1(ii)]{Baer1953}}] Let $ p$ be a prime. A $p$-element belongs to $\mathrm{Z}_\infty(G)$ iff its permutes with all $p'$-elements of a group $G$.
\end{corollary}

\section{Preliminaries}

The notation and terminology agree with \cite{s9} and \cite{s8}. We refer the reader to these
books for the results about formations.

Recall that a \emph{formation} is a class of groups which is closed under taking epimorphic images and subdirect products. A formation $\mathfrak{F}$ is called \emph{hereditary} if $H\in \mathfrak{F}$ whenever $H\leq G\in \mathfrak{F}$; \emph{saturated} if $G\in\mathfrak{F}$
whenever $G/\Phi(N)\in\mathfrak{F}$ for some normal  subgroup $N$ of $G$.

\begin{lemma}[{\cite[Lemma 2.5]{sp4}}]
  The class of all $\sigma$-nilpotent groups is a hereditary saturated formation.
\end{lemma}

The following two lemmas follow from \cite[Lemmas 6.1.6 and 6.1.7]{s9}.

\begin{lemma}\label{l3.1}  Let $\mathfrak{F}$ be a formation, $H$ and  $R$ be subgroups of   $G$ and $N\trianglelefteq G$.
\begin{enumerate}[$(1)$]
     \item If $H$ $K$-$\mathfrak{F}$-$sn\,G$, then $HN/N$ $K$-$\mathfrak{F}$-$sn\,G/N$.

\item If  $H/N $~$K$-$\mathfrak{F}$-$sn\,G/N$, then   $H$~$K$-$\mathfrak{F}$-$sn\,G$.

\item   If $H $ $K$-$\mathfrak{F}$-$sn\,R$  and  $R $~$K$-$\mathfrak{F}$-$sn\,G$, then $H$ $K$-$\mathfrak{F}$-$sn\,G$.
\end{enumerate}
\end{lemma}

\begin{lemma}\label{l3.2}
 Let $\mathfrak{F}$ be a hereditary formation, $H$ and  $R$ be subgroups of  $G$.
\begin{enumerate}[$(1)$]

\item  If   $H$~$K$-$\mathfrak{F}$-$sn\,G$, then  $H\cap R $~$K$-$\mathfrak{F}$-$sn\,R$.

\item If  $H $\,$K$-$\mathfrak{F}$-$sn\,G$ and $R $\,$K$-$\mathfrak{F}$-$sn\,G$, then   $H\cap R $\,$K$-$\mathfrak{F}$-$sn\,G$.
\end{enumerate}
\end{lemma}

The following lemma directly follows from Lemma \ref{l3.1}.

\begin{lemma}\label{lemN}
  Let $\mathfrak{F}$ be a formation, $H$ and  $R$ be  subgroups of  $G$ and $N\trianglelefteq G$. If $H$ $K$-$\mathfrak{F}$-$sn\,R$, then $HN$ $K$-$\mathfrak{F}$-$sn\,RN$.
\end{lemma}

Recall that $\mathbb{F}_p$ denotes a field with $p$ elements. The following result directly follows from \cite[B, Theorem 10.3]{s8}.

\begin{lemma}\label{10.3B}
   If $\mathrm{O}_p(G)=1$ and $G$ has a unique minimal normal subgroup, then $G$ has  a faithful irreducible   module over $\mathbb{F}_p$.
\end{lemma}

In \cite{SHEM} L.A. Shemetkov possed the problem to describe the set of  formations $\mathfrak{F}$ having the following property $$\mathfrak{F}=(G\mid\textrm{every chief factor of } G \textrm{ is }\mathfrak{F}\textrm{-central})=(G\mid G=\mathrm{Z}_\mathfrak{F}(G)).$$ This class of formations contains saturated (local) and solubly saturated (composition or Baer-local) formations and other.
Shortly we shall call formations from this class \emph{$Z$-saturated.}
In \cite{bb1}  A. Ballester-Bolinches and M. P{\'e}rez-Ramos showed that for a formation $\mathfrak{F}$ the class
$$Z\mathfrak{F}=(G\mid G=\mathrm{Z}_\mathfrak{F}(G))$$ is a formation and $\mathfrak{F}\subseteq Z\mathfrak{F}\subseteq\textbf{E}_\Phi\mathfrak{F}$.

 Let   $\mathfrak{F}$ be a hereditary formation. In \cite{vF} and \cite{wF} the classes  $\overline{w}\mathfrak{F}$ and $v^*\mathfrak{F}$  of all groups all whose Sylow and cyclic primary subgroups respectively are  $K$-$\mathfrak{F}$-subnormal were studied. From the results of these papers follows

 \begin{proposition}\label{wv}
  If $\mathfrak{F}$ is a hereditary formation, then $\overline{w}\mathfrak{F}$ and $v^*\mathfrak{F}$  are   hereditary  formations and $\mathfrak{N}\cup\mathfrak{F}\subseteq\overline{w}\mathfrak{F}\subseteq v^*\mathfrak{F}$.
 \end{proposition}

Recall that a Schmidt group $G$ is a non-nilpotent group all whose proper subgroups are nilpotent.  It is well known that $\pi(G)=\{p, q\}$ and $ G$ has a unique normal Sylow subgroup.
Recall \cite{VM} that a Schmidt $(p, q)$-group
is a Schmidt group  with a normal Sylow $p$-subgroup. An \emph{$N$-critical graph} $\Gamma_{Nc}(G)$ of a
 group $G$ \cite[Definition 1.3]{VM} is a directed graph on the vertex set $\pi(G)$ of all prime divisors
 of $|G|$ and $(p, q)$ is an edge of   $\Gamma_{Nc}(G)$ iff $G$ has  a Schmidt $(p, q)$-subgroup.
 An \emph{$N$-critical graph} $\Gamma_{Nc}(\mathfrak{X})$ of a class of groups $\mathfrak{X}$
 \cite[Definition 3.1]{VM} is a directed graph on the vertex set $\pi(\mathfrak{X})=\cup_{G\in\mathfrak{X}}\pi(G)$
 such that $\Gamma_{Nc}(\mathfrak{X})=\cup_{G\in\mathfrak{X}}\Gamma_{Nc}(G)$.

\begin{proposition}[{\cite[Theorem 5.4]{VM}}]\label{5.4}
Let $\sigma=\{\pi_i\mid i\in I\}$ be a partition of the vertex set $V(\Gamma_{Nc}(\mathfrak{X}))$ such that for $i\neq j$ there are no edges between $\pi_i$ and $\pi_j$. Then every $\mathfrak{X}$-group is
the direct product of its Hall $\pi_k$-subgroups, where $k\in \{i \in I \mid \pi(G) \cap \pi_k\neq\emptyset\}$.
\end{proposition}

\section{The proof of Theorem \ref{gb} and its corollaries}

   The proof of  Theorem  \ref{gb}  is rather complicated and require  various  preliminary results and definitions. A subgroup $U$ of  $G$ is called $\mathfrak{X}$-\emph{maximal} in $G$
provided that $(a)$ $U\in\mathfrak{X}$, and $(b)$ if $U\leq V \leq G$ and $V\in\mathfrak{X}$, then $U = V$ \cite[III, Definition 3.1]{s8}. The symbol $\mathrm{Int}_\mathfrak{X}(G)$
denotes the intersection of all $\mathfrak{X}$-maximal subgroups of $G$ \cite{h4}.


\begin{proposition}\label{l5} Let $\mathfrak{F}$ be a hereditary formation.
Then
\begin{enumerate}[$(1)$]
  \item  {\cite[Lemma 2.4]{Aivazidis2021}} $\mathrm{Z}_\mathfrak{F}(G)\cap H\leq\mathrm{Z}_\mathfrak{F}(H)$ for every subgroup $H$ of a group $G$.

  \item  $\mathrm{Z}_\mathfrak{F}(G)=\mathrm{Z}_\mathfrak{F}(\mathrm{Z}_\mathfrak{F}(G))$ for every group $G$.

  \item Assume that $H$ is an $\mathfrak{F}$-subgroup of a group $G$.    If  $\mathfrak{F}$ is $Z$-saturated, then $H\mathrm{Z}_\mathfrak{F}(G)\in\mathfrak{F}$. In particular $\mathrm{Z}_\mathfrak{F}(G)\leq\mathrm{Int}_\mathfrak{F}(G)$ for every group $ G$.
\end{enumerate}

\end{proposition}

\begin{proof} 

(2)  From $ (1)$ it follows that $\mathrm{Z}_\mathfrak{F}(G)=\mathrm{Z}_\mathfrak{F}(G)\cap \mathrm{Z}_\mathfrak{F}(G)\leq \mathrm{Z}_\mathfrak{F}(\mathrm{Z}_\mathfrak{F}(G))\leq \mathrm{Z}_\mathfrak{F}(G)$.
Thus $\mathrm{Z}_\mathfrak{F}(G)=\mathrm{Z}_\mathfrak{F}(\mathrm{Z}_\mathfrak{F}(G))$.

  (3) From $(1)$ it follows that $\mathrm{Z}_\mathfrak{F}(G)\leq \mathrm{Z}_\mathfrak{F}(G)\cap H\mathrm{Z}_\mathfrak{F}(G)\leq \mathrm{Z}_\mathfrak{F}(H\mathrm{Z}_\mathfrak{F}(G))$. Since  the  group $H\mathrm{Z}_\mathfrak{F}(G)/\mathrm{Z}_\mathfrak{F}(G)\in\mathfrak{F}$, we see that   $H\mathrm{Z}_\mathfrak{F}(G)/\mathrm{Z}_\mathfrak{F}(H\mathrm{Z}_\mathfrak{F}(G))\in\mathfrak{F}$. Hence $H\mathrm{Z}_\mathfrak{F}(G)=\mathrm{Z}_\mathfrak{F}(H\mathrm{Z}_\mathfrak{F}(G))\in Z\mathfrak{F}=\mathfrak{F}$.

 Let $M$ be an $\mathfrak{F}$-maximal subgroup of  $G$. Then $M\mathrm{Z}_\mathfrak{F}(G)\in\mathfrak{F}$. It means that $M\mathrm{Z}_\mathfrak{F}(G)=M$. Thus  $\mathrm{Z}_\mathfrak{F}(G)\leq\mathrm{Int}_\mathfrak{F}(G)$.
\end{proof}


The following result plays the key role  in the proof of Theorem \ref{gb}.

\begin{proposition}\label{p9}
Let  $\mathfrak{F}$  be a formation.
\begin{enumerate}[$(1)$]
\item $\mathrm{Z}_{Z\mathfrak{F}}(G)=\mathrm{Z}_\mathfrak{F}(G)$ holds for every group  $G$.

\item Assume that $\mathfrak{F}$ is hereditary. A subgroup  $H$  is $K$-$\mathfrak{F}$-subnormal in a group $ G$ iff it is $K$-$Z\mathfrak{F}$-subnormal in $ G$.
    \end{enumerate}
\end{proposition}

\begin{proof}
$(1)$ Let $H/K$ be a chief factor of a group $G$. Now $(H/K)\rtimes G/C_G(H/K)$ is a primitive group. It means that the $\mathfrak{F}$-hypercenter  is defined by the set of all primitive   $\mathfrak{F}$-groups.  According to \cite{bb1}   $\mathfrak{F}\subseteq Z\mathfrak{F}\subseteq \textbf{E}_\Phi\mathfrak{F}$. It means that every $Z\mathfrak{F}$-group $G$ with $\Phi(G)=1$ belongs $\mathfrak{F}$.     Thus the sets of all primitive $\mathfrak{F}$-groups and $Z\mathfrak{F}$-groups coincide. Hence  $\mathrm{Z}_{Z\mathfrak{F}}(G)=\mathrm{Z}_\mathfrak{F}(G)$.

$(2)$ Note that $Z\mathfrak{F}$ is a hereditary formation by Statement $(1)$ of Proposition \ref{l5}. Since $\mathfrak{F}$ is a hereditary formation, we see that $H$ is a  $K$-$\mathfrak{F}$-subnormal subgroup of a group $G$ if and only if there is a chain of subgroups
$ H=H_0\subseteq H_1\subseteq\dots\subseteq H_n=G$
with $H_{i-1}\trianglelefteq H_i$ or $H_{i}/\mathrm{Core}_{H_{i}}(H_{i-1})\in\mathfrak{F}$ and $H_{i-1}$ is a maximal subgroup of $H_i$  for all $i=1,\dots,n$. It means that $K$-$\mathfrak{F}$-subnormality is defined by the set of all primitive $\mathfrak{F}$-groups for a hereditary formation $ \mathfrak{F}$.  As we have already mentioned  the sets of all primitive $\mathfrak{F}$-groups and $Z\mathfrak{F}$-groups coincide. Thus a subgroup   is $K$-$\mathfrak{F}$-subnormal in a group $ G$ iff it is $K$-$Z\mathfrak{F}$-subnormal in $ G$.
\end{proof}

The next step in the proof of Theorem \ref{gb} is to characterize the intersections  $S_\mathfrak{F}(G)$ and $C_\mathfrak{F}(G)$ of all weak $K$-$\mathfrak{F}$-subnormalizers of all Sylow and all cyclic primary subgroups of $G$ respectively.

\begin{proposition}\label{p1} Let $\mathfrak{F}$ be a hereditary formation.
\begin{enumerate}[$(1)$]
 \item $S_\mathfrak{F}(G)$ is the largest subgroup among normal subgroups $N$ of $G$ with  $P$ $K$-$\mathfrak{F}$-$sn\,PN$ for every Sylow subgroup $P$ of $G$.

   \item$C_\mathfrak{F}(G)$ is the largest subgroup among normal subgroups $N$ of $G$ with   $C$ $K$-$\mathfrak{F}$-$sn\,CN$ for every cyclic primary subgroup $C$ of $G$.
 \end{enumerate}
 \end{proposition}

 \begin{proof}
   $(1)$ Let $N\trianglelefteq G$ with  $P$ $K$-$\mathfrak{F}$-$sn\,PN$ for every Sylow subgroup   $P$ of $G$. If $S$ is a weak  $K$-$\mathfrak{F}$-subnormalizer of  $P$ in $G$, then $PN$  $K$-$\mathfrak{F}$-$sn\, SN$  by Lemma \ref{lemN}. Hence $P$  $K$-$\mathfrak{F}$-$sn\, SN$ by   $(3)$ of Lemma \ref{l3.1}.  Now $SN=S$ by the definition of a weak   $K$-$\mathfrak{F}$-subnormalizer. Thus $N\leq S_\mathfrak{F}(G)$.

   From the other hand, since  $\mathfrak{F}$ is a hereditary formation and $PS_\mathfrak{F}(G)$ lies in every weak $K$-$\mathfrak{F}$-subnormalizer of every Sylow subgroup $P$ of $G$, we see that $P$ $K$-$\mathfrak{F}$-$sn\,PS_\mathfrak{F}(G)$ for every Sylow subgroup   $P$ of $G$ by Lemma \ref{l3.2}.  Thus $S_\mathfrak{F}(G)$ is the largest normal subgroup  $N$ of $G$ with $P$ $K$-$\mathfrak{F}$-$sn\,PN$ for every Sylow subgroup $P$ of $G$.

  The proof of $(2)$ is the same. \end{proof}

The connections between the previous steps are shown in the following proposition:

\begin{proposition}\label{p2} Let $\mathfrak{F}$ be a hereditary formation. Then $\overline{w}\mathfrak{F}$ and $v^*\mathfrak{F}$  are   hereditary $Z$-saturated formations and $\mathrm{Int}_{\overline{w}\mathfrak{F}}(G)=S_\mathfrak{F}(G)\leq C_\mathfrak{F}(G)=\mathrm{Int}_{v^*\mathfrak{F}}(G)$ holds for every group $G$.
 \end{proposition}

\begin{proof} Note that $v^*\mathfrak{F}$ and $\overline{w}\mathfrak{F}$ are  hereditary formations by Proposition \ref{wv}.
Assume that   $\overline{w}\mathfrak{F}$ is not a   $Z$-saturated formation. Let chose a minimal order group   $G$ from  $Z(\overline{w}\mathfrak{F})\setminus \overline{w}\mathfrak{F}$. From Proposition \ref{l5} it follows that   $Z\overline{w}\mathfrak{F}$ is a hereditary formation. So  $G$ is $\overline{w}\mathfrak{F}$-critical. Now $|\pi(G)|>1$ by    Proposition \ref{wv}. From  $\overline{w}\mathfrak{F}\subset Z\overline{w}\mathfrak{F}\subseteq \textbf{E}_\Phi \overline{w}\mathfrak{F}$ it follows that   $\Phi(G)\neq 1$ and $G/\Phi(G)\in \overline{w}\mathfrak{F}$. Let   $P$ be a Sylow subgroup of $G$. Then $P\Phi(G)< G$ and $P\Phi(G)\in \overline{w}\mathfrak{F}$. Hence $P$ $K$-$\mathfrak{F}$-$sn\,P\Phi(G)$. From $G/\Phi(G)\in \overline{w}\mathfrak{F}$ it follows that $P\Phi(G)/\Phi(G)$ $K$-$\mathfrak{F}$-$sn\,G/\Phi(G)$. Therefore $P\Phi(G)$ $K$-$\mathfrak{F}$-$sn\,G$. Thus $P$ $K$-$\mathfrak{F}$-$sn\,G$. It means that  $G\in \overline{w}\mathfrak{F}$, a contradiction. Thus  $\overline{w}\mathfrak{F}$ is a  $Z$-saturated formation. The proof for $v^*\mathfrak{F}$ is the same.

Note that   $\mathfrak{N}\subseteq v^*\mathfrak{F}$  by Proposition~\ref{wv}. Hence $C\mathrm{Int}_{v^*\mathfrak{F}}(G)\in v^*\mathfrak{F}$ for every cyclic primary subgroup   $C$ of $G$.  Therefore $C$ $K$-$\mathfrak{F}$-$sn\,C\mathrm{Int}_{v^*\mathfrak{F}}(G)$  for every cyclic primary subgroup   $C$ of $G$.  Thus $\mathrm{Int}_{v^*\mathfrak{F}}(G)\leq C_\mathfrak{F}(G)$  by $(2)$ of Proposition~\ref{p1}.

From the other hand let  $M$ be  a $v^*\mathfrak{F}$-maximal subgroup of  $G$ and $C$ be a cyclic primary subgroup of  $MC_\mathfrak{F}(G)$.
Since $MC_\mathfrak{F}(G)/C_\mathfrak{F}(G)\in v^*\mathfrak{F}$, we see that $C_\mathfrak{F}(G)C/C_\mathfrak{F}(G)$    $K$-$\mathfrak{F}$-$sn\,MC_\mathfrak{F}(G)/C_\mathfrak{F}(G)$. Hence $C_\mathfrak{F}(G)C$    $K$-$\mathfrak{F}$-$sn\,MC_\mathfrak{F}(G)$ by $(2)$  of Lemma  \ref{l3.1}. Note that   $C$   $K$-$\mathfrak{F}$-$sn\,C_\mathfrak{F}(G)C$ by Proposition  \ref{p1}. So $C$   $K$-$\mathfrak{F}$-$sn\,MC_\mathfrak{F}(G)$ by $(3)$ of Lemma \ref{l3.1}. Thus $MC_\mathfrak{F}(G)\in v^*\mathfrak{F}$ by the definition of  $v^*\mathfrak{F}$. Hence $MC_\mathfrak{F}(G)=M$. Therefore $C_\mathfrak{F}(G)\leq\mathrm{Int}_{v^*\mathfrak{F}}(G)$. Thus $\mathrm{Int}_{v^*\mathfrak{F}}(G)= C_\mathfrak{F}(G)$.
The proof of that equality  $\mathrm{Int}_{\overline{w}\mathfrak{F}}(G)= S_\mathfrak{F}(G)$ holds in every group is the same.

Since every cyclic primary subgroup is subnormal in some Sylow subgroup, we see that  $P$ $K$-$\mathfrak{F}$-$sn\,PS_\mathfrak{F}(G)$ for every cyclic primary subgroup $P$ of $G$. So $S_\mathfrak{F}(G)\leq C_\mathfrak{F}(G)$ holds for every group $G$ by Proposition~\ref{p1}.
\end{proof}

\begin{proof}[Proof of Theorem \ref{gb}]
$(1)\Rightarrow(2)$. Since $\mathfrak{F}\subseteq \overline{w}\mathfrak{F}$ by Proposition \ref{wv}, we see that $\mathrm{Z}_\mathfrak{F}(G)\leq\mathrm{Z}_{\overline{w}\mathfrak{F}}(G)$ for every group $G$. Note that $\mathrm{Z}_{\overline{w}\mathfrak{F}}(G)\leq\mathrm{Int}_{\overline{w}\mathfrak{F}}(G)$  for every group $G$ by $(3)$ of Proposition  \ref{l5} and Proposition \ref{p2}. According to Proposition \ref{p2},   $S_\mathfrak{F}(G)=\mathrm{Int}_{\overline{w}\mathfrak{F}}(G)$ and
  $S_\mathfrak{F}(G)\leq C_\mathfrak{F}(G)$ for every group $G$. From these and   $(1)$ it follows that
  $$\mathrm{Z}_\mathfrak{F}(G)\leq\mathrm{Z}_{\overline{w}\mathfrak{F}}(G)\leq\mathrm{Int}_{\overline{w}\mathfrak{F}}(G)
  =S_\mathfrak{F}(G)\leq C_\mathfrak{F}(G)=\mathrm{Z}_\mathfrak{F}(G)$$
 for every group $G$.  Thus $\mathrm{Z}_\mathfrak{F}(G) =S_\mathfrak{F}(G)$ for every group $G$.

$(2)\Rightarrow(3)$. The proof consists of the following steps:

$(a)$ \emph{We may assume that $\mathfrak{N}\subseteq\mathfrak{F}$ is $Z$-saturated}.

According to Proposition \ref{p9} Statements $(2)$ and $(3)$   mean the same for   $\mathfrak{F}$ and $Z\mathfrak{F}$. Note that $Z\mathfrak{F}=Z(Z\mathfrak{F})$ by Proposition \ref{p9}. \textbf{\emph{Therefore without lose of generality we may assume that  $\mathfrak{F}$ is $Z$-saturated in the proof of $(2)\Rightarrow(3)$}}. Since in every nilpotent group every Sylow subgroup is subnormal and $Z\mathfrak{F}=\mathfrak{F}$ we see that   $\pi(\mathfrak{F})=\mathbb{P}$ and $\mathfrak{N}\subseteq\mathfrak{F}$.

 $(b)$ \emph{Assume that a group $G$ has faithful irreducible module $L$  over $\mathbb{F}_p$,    $T=L\rtimes G$  and $L\leq S_\mathfrak{F}(T)$. Then $G\in\mathfrak{F}$.}

Note that $L\leq S_\mathfrak{F}(G)=\mathrm{Z}_{\mathfrak{F}}(T)$. Hence $L\rtimes (T/C_T(L))\in\mathfrak{F}$. Thus $G\simeq T/C_T(L)\in\mathfrak{F}$, the contradiction.

$(c)$ \emph{Let $\pi(p)=\{q\in\mathbb{P}\,|\,(p, q)\in\Gamma_{Nc}(\mathfrak{F})\}\cup\{p\}$. Then $\mathfrak{F}$ contains every $q$-closed $\{p, q\}$-group for every  $q\in\pi(p)$.}

Assume the contrary. Let $G$ be a minimal order counterexample. Since $\mathfrak{F}$ and the class of all $q$-closed groups are hereditary formations, we see that  $G$ is an $\mathfrak{F}$-critical group, $G$  has a unique minimal normal subgroup $N$ and $G/N\in\mathfrak{F}$. Let $P$ be a Sylow $p$-subgroup of $G$.   If $NP<G$, then $NP\in\mathfrak{F}$. Hence $P$ $K$-$\mathfrak{F}$-$sn\,PN$  and $PN/N$ $K$-$\mathfrak{F}$-$sn\,G/N$. From Lemma \ref{l3.1} it follows that $P$ $K$-$\mathfrak{F}$-$sn\,G$. Since $G$  is a $q$-closed $\{p, q\}$-group, we see that every Sylow subgroup of $G$ is $K$-$\mathfrak{F}$-subnormal. So $G\in Z\mathfrak{F}=\mathfrak{F}$, a contradiction.

Now $N$ is a Sylow $q$-subgroup  and $\mathrm{O}_p(G)=1$. By Lemma \ref{10.3B} $G$ has a faithful irreducible  module $L$ over $\mathbb{F}_p$. Let $T=L\rtimes G$.  Therefore for every chief factor $H/K$ of $NL$  a group    $(H/K)\rtimes C_{NL}(H/K)$ is isomorphic to one of the following groups  $Z_p, Z_q$ and a Schmidt $(p, q)$-group with the trivial Frattini subgroup. Note that all these groups belong $\mathfrak{F}$. So  $NL\in Z\mathfrak{F}=\mathfrak{F}$. Note that $L\leq \mathrm{O}_p(T)$. Hence  $L\leq S_\mathfrak{F}(T)$ by Proposition \ref{p1}.
Thus $G\in\mathfrak{F}$ by $(b)$, a contradiction.

From $(c)$ it follows that

 $(d)$ \emph{$\Gamma_{Nc}(\mathfrak{F})$ is undirected, i.e $(p, q)\in\Gamma_{Nc}(\mathfrak{F})$ iff $(q, p)\in\Gamma_{Nc}(\mathfrak{F})$.}

$(e)$ \emph{Let   $p, q$ and $r$ be different primes.  If $(p, r), (q, r)\in\Gamma_{Nc}(\mathfrak{F})$, then   $(p, q)\in\Gamma_{Nc}(\mathfrak{F})$}.

Note that the cyclic group $Z_q$ of order $q$ has a faithful irreducible    module $P$ over $\mathbb{F}_p$ by Lemma  \ref{10.3B}. Let $G=P\rtimes Z_q$.  Then $G$ has a faithful irreducible     module $R$ over $\mathbb{F}_r$ by Lemma  \ref{10.3B}. Let $T=R\rtimes G$. From $(c)$ it follows that   $\mathfrak{F}$-contains  all $r$-closed $\{p, r\}$-groups and $\{q, r\}$-groups. Hence $R\leq S_\mathfrak{F}(T)$  by Proposition~\ref{p1}. Thus $G\in\mathfrak{F}$ by $(b)$. Note that $G$  is a Schmidt $(p, q)$-group. It means that $(p, q)\in\Gamma_{Nc}(\mathfrak{F})$ by the definition of $N$-critical graph.

$(f)$   \emph{$\mathfrak{F}=\mathfrak{N}_\sigma$ for some partition $\sigma$ of $\mathbb{P}$.  }

From $(d)$ and $(e)$ it follows that $\Gamma_{Nc}(\mathfrak{F})$ is a disjoint union of  complete (directed) graphs $\Gamma_i$, $i\in I$. Let $\pi_i=V(\Gamma_i)$. Then $\sigma=\{\pi_i\,|\,i\in I\}$ is a partition of $\mathbb{P}$. From Proposition \ref{5.4} it follows that every $\mathfrak{F}$-group $G$ has normal Hall $\pi_i$-subgroups for every $i\in I$ with $\pi_i\cap\pi(G)\neq\emptyset$. So $G$ is $\sigma$-nilpotent. Hence        $\mathfrak{F}\subseteq\mathfrak{N}_\sigma$.

Let show that the class
$\mathfrak{G}_{\pi_i}$ of all $\pi_i$-groups is a subset of  $\mathfrak{F}$ for every $i\in I$.
It is true  if $|\pi_i|=1$. Assume now $|\pi_i|>1$.  Suppose the contrary and let a group  $G$ be a minimal order group from $\mathfrak{G}_{\pi_i}\setminus\mathfrak{F}$. Then $G$ has a unique minimal normal subgroup, $\pi(G)\subseteq\pi_i$ and $|\pi(G)|>1$. Note that $\mathrm{O}_q(G)=1$ for some $q\in\pi(G)$. Hence $G$ has  a faithful irreducible     module $N$ over $\mathbb{F}_q$ by Lemma \ref{10.3B}. Let $T=N\rtimes G$. Hence $NP\in\mathfrak{F}$ for every Sylow subgroup $P$ of  $T$ by $(c)$. Now $N\leq S_\mathfrak{F}(T)$ by Proposition \ref{p1}.
So $G\in\mathfrak{F}$ by $(b)$, the contradiction.

Since a formation is closed under taking direct products, we see that $\mathfrak{N}_\sigma\subseteq\mathfrak{F}$.
Thus $\mathfrak{F}=\mathfrak{N}_\sigma$.

$(3)\Rightarrow(1)$. Recall that the class of all $\sigma$-nilpotent groups is saturated. Hence it is $Z$-saturated.
According to Proposition \ref{p9} Statements $(3)$ and $(1)$   mean the same for   $\mathfrak{F}$ and $Z\mathfrak{F}$. Hence we may assume that $\mathfrak{F}=\mathfrak{N}_\sigma$ for some partition $\sigma=\{\pi_i\mid i\in I\}$ of $\mathbb{P}$.
Then $\mathfrak{\mathfrak{N}_\sigma}$ has the lattice property for $K$-$\mathfrak{F}$-subnormal subgroups  (see \cite[Lemma 2.6(3)]{sp4} or \cite[Chapter 3]{s9}).
According to \cite[Theorem B and Corollary E.2]{vF} $v^*\mathfrak{F}=\mathfrak{F}$. By \cite[Theorem A and Proposition 4.2]{h4}    $\mathrm{Int}_\mathfrak{F}(G)=\mathrm{Z}_\mathfrak{F}(G)$ holds for every group $G$. By  Proposition \ref{p2}, $C_\mathfrak{F}(G)=\mathrm{Int}_{v^*\mathfrak{F}}(G)$ for every group $G$. Thus
$C_\mathfrak{F}(G)=\mathrm{Int}_{v^*\mathfrak{F}}(G)=\mathrm{Int}_{\mathfrak{F}}(G)=\mathrm{Z}_\mathfrak{F}(G) $ for every group $G$.

$(3)\Rightarrow (4)$ Statement $(3)$ means that $Z\mathfrak{F}=\mathfrak{N}_\sigma$ and  $\pi(\mathfrak{F})=\pi(Z\mathfrak{F})=\mathbb{P}$. From $ \mathfrak{F}\subseteq Z\mathfrak{F}$ it follows that $ \mathfrak{F}=\times_{i\in I} \mathfrak{F}_{\pi_i}$  where $\mathfrak{F}_{\pi_i}$ is a hereditary formation with $\pi(\mathfrak{F}_{\pi_i})=\pi_i$.

Assume that $\pi_i\in\sigma$ and $ |\pi_i|\geq 2$.
Let choose a minimal order $\pi_i$-group $G$ from $Z\mathfrak{F}\setminus\mathfrak{F}_{\pi_i}$. Since $Z\mathfrak{F}=\mathfrak{N}_\sigma$ and $\mathfrak{F}_{\pi_i}=\mathfrak{F}\cap\mathfrak{G}_{\pi_i}$ are   formations, we see that  $G$ has a unique minimal normal subgroup $N$. From $|\pi_i|\geq 2$ it follows that there exists $p\in{\pi_i}$ such that $N$ is not a $p$-group.
 Therefore $G$ has  a faithful irreducible   module $V$ over $\mathbb{F}_p$ by Lemma  \ref{10.3B}. Let $T=V\rtimes G$. Since $T$ is a $\pi_i$-group, $ T\in  \mathfrak{N}_\sigma= Z\mathfrak{F}$. Hence $R=V\rtimes (T/C_T(V))\in \mathfrak{F} \cap\mathfrak{G}_{\pi_i}=\mathfrak{F}_{\pi_i}$ and $T/C_T(V)\simeq G$. Now $G\in\mathfrak{F}_{\pi_i}$ as a quotient group of $R$, a contradiction. It means that $\mathfrak{F}\cap \mathfrak{G}_{\pi_i}=Z\mathfrak{F}\cap \mathfrak{G}_{\pi_i}=\mathfrak{G}_{\pi_i}$.

$(4)\Rightarrow (3)$ Assume that $ \mathrm{Z}_\mathfrak{F}(G)\neq\mathrm{Z}_{\mathfrak{N}_\sigma}(G)$ for some group $G$. It means that there exists a primitive $\mathfrak{N}_\sigma$-group $H$ with $H\not\in\mathfrak{F}$. Since $H$ is       a primitive $\mathfrak{N}_\sigma$-group, we see that $H$ is a $\pi_i$-group for some $i\in I$. If $|\pi_i|\geq 2$, then $ H\in\mathfrak{G}_{\pi_i}\subseteq\mathfrak{F}$, a contradiction. Hence $|\pi_i|=1$. So $H$ is a $p$-group for some $p\in\mathbb{P}$. Therefore $H$ is a cyclic group of order $p$. Thus $H\in\mathfrak{F}$, the final contradiction.
 \end{proof}

\begin{proof}[Proof of Corollary \ref{gs}]
  Let $ D$ be the intersection of normalizers in $G$ of all subgroups from $\mathcal{M}$. From $(a)$ it follows that $D\trianglelefteq G$.   Let $P$ be a Sylow subgroup of $G$ and $H$ be a subgroup from $\mathcal{M}$ with $P\leq H$. Note that $H\in \mathfrak{N}_\sigma$. Now $P$ $K$-$\mathfrak{N}_\sigma$-$sn\,H\trianglelefteq HD$. So $P$ $K$-$\mathfrak{N}_\sigma$-$sn\, HD$. Hence    $P$ $K$-$\mathfrak{N}_\sigma$-$sn\, PD$ by Lemma \ref{l3.2}. It means that $D$ $K$-$\mathfrak{N}_\sigma$-subnormalizes all Sylow subgroups of $G$. Thus $ D\leq S_{\mathfrak{N}_\sigma}(G)$ by Proposition \ref{p1}.

  From the proof of Theorem \ref{gb} it follows that $S_{\mathfrak{N}_\sigma}(G)=\mathrm{Z}_{\mathfrak{N}_\sigma}(G)=\mathrm{Int}_{\mathfrak{N}_\sigma}(G)$. Let $H\in\mathcal{M}$. Now $HS_{\mathfrak{N}_\sigma}(G)\in \mathfrak{N}_\sigma$. Since $H$ is a $\pi_i$-maximal subgroup of $G$, $H$ is a $\pi_i$-maximal subgroup of $HS_{\mathfrak{N}_\sigma}(G)$. It means that $H\trianglelefteq HS_{\mathfrak{N}_\sigma}(G)$. So $S_{\mathfrak{N}_\sigma}(G)$ normalizes all subgroups from $\mathcal{M}$.  Hence $ S_{\mathfrak{N}_\sigma}(G)\leq D$. Thus $D=S_{\mathfrak{N}_\sigma}(G)=\mathrm{Z}_{\mathfrak{N}_\sigma}(G)$ by Theorem \ref{gb}.
\end{proof}

\begin{proof}[Proof of Corollary \ref{cor15}]
From Corollary \ref{corollary2} it follows that every  $\pi_i$-element of $\mathrm{Z}_{\mathfrak{N}_\sigma}(G)$ permutes with every $\pi_i'$-element of $ G$.

  Let $A$ be the set of all $\pi_i$-elements of  $G$ that permute   with all $\pi_i'$-elements of $ G$ and $H=\langle A\rangle$. So all elements of $H$ permute   with all $\pi_i'$-elements of $ G$. Since $\mathrm{O}^{\pi_i}(H)$ is generated by all $\pi_i'$-elements of $H$, we see that $\mathrm{O}^{\pi_i}(H)\leq \mathrm{Z}(H)$. Hence all $\pi_i'$-elements of  $\mathrm{O}^{\pi_i}(H)$ form a subgroup. So $ \mathrm{O}^{\pi_i}(H)$ is a $\pi_i'$-group. Let $ K$ be a $ \mathfrak{G}_{\pi_i}$-projector of $ H$. Then $ K\mathrm{O}^{\pi_i}(H)=H$. So $ K\trianglelefteq H$. It means that $\pi_i$-elements of $ H$ form a subgroup. Thus $ H=A$.

  Now $ A$ is a normal $ \pi_i$-subgroup of $G$. Hence it lies in every $ \pi_i$-maximal subgroup of $G$. Note that $A$ lies in the normalizer of every $\pi_i'$-subgroup  by its definition. Thus $A\leq\mathrm{Z}_{\mathfrak{N}_\sigma}(G)$ by Corollary \ref{corollary2}.
\end{proof}

\section{Applications}

 R. Baer \cite{Baer1953} proved that the hypercenter of a group coincides with the intersection of all its maximal nilpotent subgroups.
 L.\,A. Shemetkov possed a question at the  Gomel Algebraic Seminar in 1995 that can be formulated in the following way:
 For what non-empty  (normally) hereditary (solubly)
 saturated formations $\mathfrak{F}$ does the intersection of all $\mathfrak{F}$-maximal subgroups coincides with the $ \mathfrak{F}$-hypercenter in every group?
 A.\,N. Skiba    \cite{h4} answered on this question for hereditary saturated formations $\mathfrak{F}$ (for the soluble case, see also J.\,C. Beidleman  and H. Heineken \cite{h3}).
From Theorem \ref{gb} follows a solution of this question for a family of hereditary not
necessary saturated formations.

\begin{theorem}\label{Thm2}
  Let $\mathfrak{F}$ be a hereditary formation.

  \begin{enumerate}[$(1)$]

   \item $\mathfrak{F}=\overline{w}\mathfrak{F}$ if and only if $S_\mathfrak{F}(G)=\mathrm{Int}_\mathfrak{F}(G)$ holds for every group.

  \item $\mathfrak{F}=v^*\mathfrak{F}$ if and only if $C_\mathfrak{F}(G)=\mathrm{Int}_\mathfrak{F}(G)$ holds for every group.

  \item
  Assume that $\mathfrak{F}=\overline{w}\mathfrak{F}$   or $ \mathfrak{F}=v^*\mathfrak{F}$. Then $\mathrm{Z}_\mathfrak{F}(G)=\mathrm{Int}_\mathfrak{F}(G)$ holds for every group if and only if there is a partition   $\sigma$ of $\mathbb{P}$ such that $\mathfrak{F}$ is the class of all  $\sigma$-nilpotent groups.
  \end{enumerate}
\end{theorem}

\begin{proof}
From Proposition \ref{p2} it follows that  $S_\mathfrak{F}(G)=\mathrm{Int}_{\overline{w}\mathfrak{F}}(G)$. Now $(1)$ follows from the fact that  $\mathrm{Int}_{\mathfrak{F}}(G)= \mathrm{Int}_{\overline{w}\mathfrak{F}}(G)$ holds for every group if and only if $\mathfrak{F}=\overline{w}\mathfrak{F}$. The proof of $(2)$ is the same.

$(3)$   Assume that $\mathfrak{F}=\overline{w}\mathfrak{F}$. Now $\mathfrak{F}$ is $Z$-saturated by Proposition \ref{p2} and $ \mathrm{Int}_\mathfrak{F}(G)=\mathrm{Int}_{\overline{w}\mathfrak{F}}(G)$ holds for every group $G$. From Proposition \ref{p2} it follows that $ S_\mathfrak{F}(G)=\mathrm{Int}_{\overline{w}\mathfrak{F}}(G)$ holds for every group $G$. Now $ \mathrm{Int}_\mathfrak{F}(G)=\mathrm{Z}_{\mathfrak{F}}(G)$ holds for every group if and only if $ S_\mathfrak{F}(G)=\mathrm{Z}_{\mathfrak{F}}(G)$ holds for every group $G$.  From $(3)$ of Theorem \ref{gb} it follows that the last equality holds for every group if and only if there is a partition $ \sigma$ of $\mathbb{P}$ such that $Z\mathfrak{F}=\mathfrak{N}_\sigma$. Hence $\mathfrak{F}=\mathfrak{N}_\sigma$.   From this theorem is also follows that $\mathfrak{N}_\sigma=\overline{w}\mathfrak{N}_\sigma$.

  The proof of $(3)$  for   $\mathfrak{F}=v^*\mathfrak{F}$ is the same.
\end{proof}

\begin{remark}
  There is a rather important family of not necessary saturated hereditary formations $\mathfrak{F}$ with $v^*\mathfrak{F}=\mathfrak{F}$ and  $\overline{w}\mathfrak{F}=\mathfrak{F}$. Recall that a formation $ \mathfrak{F}$ has the Shemetkov property if every $\mathfrak{F}$-critical group is either a Schmidt group of a cyclic group of prime order. The family of hereditary formations with the Shemetkov property   contains non-saturated formations (see  \cite[Chapter 6.4]{s9}).
   For example let $\mathfrak{F}$ be a class of groups all whose Schmidt subgroups are Schmidt $(p, q)$-groups for $(p, q)\in\{(2,3), (3,2), (5, 2)\}$. Then $ \mathfrak{F}$ has the Shemetkov property by \cite[Theorem 3.5]{VM} and $\pi(\mathfrak{F})=\mathbb{P}$. Let $G$ be the alternating group of degree 5. Hence $G\in\mathfrak{F}$.  According to \cite{Griess1978} there is a Frattini $\mathbb{F}_3G$-module $T$ which is faithful for $G$. By the  Gasch\"{u}tz theorem (see \cite[Appendix $\beta$]{s8}), there
exists a Frattini extension  $T\rightarrowtail R\twoheadrightarrow G$
such that $T\stackrel {G}{\simeq} \Phi(R)$ and $R/\Phi(R)\simeq G$.
 Let $ K/\Phi(R)$ be a cyclic subgroup of $G/\Phi(G)$ of order 5. Since $ T$ is faithful for $ G$, we see that $ K$ is a non-nilpotent group with a normal Sylow $ 3$-subgroup. Hence its contains a Schmidt $(5, 3)$-subgroup. It means that $ G\not\in\mathfrak{F}$, i.e. $ \mathfrak{F}$ is not saturated.

   As follows from  \cite{vF,wF} and \cite[Corollaries 3.9 and 3.10]{Murashka2018} $v^*\mathfrak{F}=\mathfrak{F}$ and  $\overline{w}\mathfrak{F}=\mathfrak{F}$ for every hereditary formation $\mathfrak{F}$ with the Shemetkov property and $\pi(\mathfrak{F})=\mathbb{P}$.\end{remark}

Let give another application of Theorem \ref{gb}.
Recall that a formation $\mathfrak{F}$ is called regular \cite{Nemmi2020}, if for every group $G$ holds $$\mathcal{I}_\mathfrak{F}(G)=\{x\in G\mid \langle x, y\rangle\in\mathfrak{F}\,\,\forall y\in G\}=\mathrm{Int}_\mathfrak{F}(G).$$
The regular formations of soluble groups were studied in \cite{Nemmi2020}. Here we give examples of such formations of non-necessary soluble groups.

Recall (see \cite{Nemmi2020}) that the non-$\mathfrak{F}$-graph  $\Gamma_\mathfrak{F}(G)$ of a group $G$ is the graph whose vertex set is $G\setminus\mathcal{I}_\mathfrak{F}(G)$ and two vertices $x$ and $y$ are connected if $\langle x, y\rangle\not\in\mathfrak{F}$. This type of graphs can be traced back to P. Erd\H{o}s who considered non-commuting (non-abelian) graph.  A. Abdollahi and M. Zarrin  \cite{Abdollahi2010}  asked to find the  bounds for diameters of non-nilpotent graphs. The final answer on this question was obtained by A. Lucchini  and D. Nemmi \cite{Lucchini2020}.

\begin{theorem}
The formation of all $\sigma$-nilpotent groups is regular and  $\mathcal{I}_{\mathfrak{N}_\sigma}(G)=\mathrm{Z}_{\mathfrak{N}_\sigma}(G)$ holds for every group $G$. Moreover the graph $\Gamma_{\mathfrak{N}_\sigma}(G)$ is connected and  $\mathrm{diam}(\Gamma_{\mathfrak{N}_\sigma}(G))\leq 3$ for every group $G$.
\end{theorem}

\begin{proof} Let $x\in G$. Denote by $G_p$ and $x_p$ a Sylow $p$-subgroup of $G$ and    $x^{|G|/|G_p|}$ respectively. Note that if $\langle x_p, y\rangle\not\in\mathfrak{N}_\sigma$, then $\langle x, y\rangle\not\in\mathfrak{N}_\sigma$.

$(1)$ \emph{$\mathfrak{N}_\sigma$ is regular and  $\mathcal{I}_{\mathfrak{N}_\sigma}(G)=\mathrm{Z}_{\mathfrak{N}_\sigma}(G)$ holds for every group $G$}

  Let $y\in G$. Then $\langle y\rangle\in\mathfrak{N}_\sigma$. It means that $\langle y\rangle\mathrm{Z}_{\mathfrak{N}_\sigma}(G)\in\mathfrak{N}_\sigma$. Hence $\langle x, y\rangle\in\mathfrak{N}_\sigma$ for all $x\in \mathrm{Z}_{\mathfrak{N}_\sigma}(G)$ and $y\in G$. It means that $\mathrm{Z}_{\mathfrak{N}_\sigma}(G)\subseteq \mathcal{I}_{\mathfrak{N}_\sigma}(G)$.

  Let $x\in \mathcal{I}_\mathfrak{F}(G)$. Note that  $x=\prod_{p\in\pi(G)}x_p$. From $\langle x_p, y\rangle\leq \langle x,y\rangle\in\mathfrak{N}_\sigma$ it follows that $x_p\in \mathcal{I}_{\mathfrak{N}_\sigma}(G)$ for all $p\in \pi(G)$.

  Let $q\in\pi(G)$. Since $\sigma$ is a partition of $\mathbb{P}$, there exists a unique $\pi_i\in\sigma$ with    $q\in\pi_i$. Let $y$ be a $\pi_i'$-element of $G$. Now $\langle x_q, y\rangle\in\mathfrak{N}_\sigma$. It means that $x_qy=yx_q$. So a $\pi_i$-element $x_q$ permutes with all $\pi_i'$-elements of $G$. Thus $x_q\in \mathrm{Z}_{\mathfrak{N}_\sigma}(G)$ by Corollary \ref{cor15}. Therefore $x\in \mathrm{Z}_{\mathfrak{N}_\sigma}(G)$. So $\mathcal{I}_{\mathfrak{N}_\sigma}(G)\subseteq\mathrm{Z}_{\mathfrak{N}_\sigma}(G)$. Hence $\mathcal{I}_{\mathfrak{N}_\sigma}(G)=\mathrm{Z}_{\mathfrak{N}_\sigma}(G)=\mathrm{Int}_{\mathfrak{N}_\sigma}(G)$.
  Thus $\mathfrak{N}_\sigma$ is regular.

$(2)$ \emph{$\Gamma_{\mathfrak{N}_\sigma}(G)$ is connected and  $\mathrm{diam}(\Gamma_{\mathfrak{N}_\sigma}(G))\leq 3$ for every group $G$.}

If $ |G\setminus\mathcal{I}_{\mathfrak{N}_\sigma}(G)|<2$, then there is nothing to prove. So we may assume that $ |G\setminus\mathcal{I}_{\mathfrak{N}_\sigma}(G)|\geq2$.  Assume that $G$ is a counterexample to $(2)$.  Hence there are elements $ x, y\in G$ such that they are not connected or the lengths of all paths connecting them are greater than 3.

If $x_p\in \mathcal{I}_{\mathfrak{N}_\sigma}(G)$ for all $p\in\pi(G)$, then $x=\prod_{p\in\pi(G)}x_p\in \mathrm{Z}_{\mathfrak{N}_\sigma}(G)=\mathcal{I}_{\mathfrak{N}_\sigma}(G)$, a contradiction. It means that there exist $p, q\in\pi(G)$ with $x_p, y_q\not\in\mathcal{I}_{\mathfrak{N}_\sigma}(G)$.  Hence there exist $\pi_i, \pi_j\in \sigma$, $\pi_i$-element $w$ and $\pi_j$-element $z$ with $ p\not\in\pi_i$, $q\not\in\pi_j$,   $\langle x_p, w\rangle\not\in\mathfrak{N}_\sigma$ and $\langle y_q, z\rangle\not\in\mathfrak{N}_\sigma$.

 If $ \langle w, z\rangle\not\in\mathfrak{N}_\sigma$, then $(x,w,z,y)$ is the path connecting $x$ and $y$ and its length is not greater than 3, a contradiction.
Now $ \langle w, z\rangle\in\mathfrak{N}_\sigma$. Assume that $i\neq j$. So $ wz=zw$ and $\langle zw\rangle=\langle z, w\rangle$. Now $(x,wz,y)$ is the path connecting $x$ and $y$ of length  2, a contradiction.
So $i=j$.
If $\langle x_p, z\rangle\not\in\mathfrak{N}_\sigma$, then $(x,z,y)$ is the path connecting $x$ and $y$ of length   2, a contradiction.  Hence  $\langle x_p, z\rangle\in\mathfrak{N}_\sigma$. Since $ p\not\in \pi_i=\pi_j$, we see that $x_pz=zx_p$ and  $\langle zx_p\rangle=\langle z, x_p\rangle$. Now $(x,w, x_pz,y)$ is the path connecting $x$ and $y$ and its length is not greater than 3, the final contradiction.
\end{proof}

\begin{corollary}[{\cite[Theorem 1.1]{Lucchini2020}}]
$\Gamma_{\mathfrak{N}}(G)$ is connected and  $\mathrm{diam}(\Gamma_{\mathfrak{N}}(G))\leq 3$ for every group $G$.
\end{corollary}

\begin{corollary}[{\cite[Theorem 5.1]{Abdollahi2010}}]
$\Gamma_{\mathfrak{N}}(G)$ is connected and  $\mathrm{diam}(\Gamma_{\mathfrak{N}}(G))\leq 6$ for every group $G$.
\end{corollary}

\end{document}